\documentclass[a4paper,twoside,12pt]{article}
\usepackage{amssymb,amsmath,amsthm,latexsym}
\usepackage{amsfonts}
\usepackage{graphicx,caption,subcaption}
\usepackage[pdftex,bookmarks,colorlinks=false]{hyperref}
\usepackage[hmargin=1.2in,vmargin=1.2in]{geometry}
\usepackage[mathscr]{euscript}

\usepackage[auth-lg,affil-sl]{authblk}
\usepackage{float}

\def\ni{\noindent}
\def\N{\mathbb{N}}

\newcommand{\J}{\mathscr{J}}

\newtheorem{theorem}{Theorem}[section]
\newtheorem{definition}[theorem]{Definition}
\newtheorem{lemma}[theorem]{Lemma}
\newtheorem{proposition}[theorem]{Proposition}
\newtheorem{corollary}[theorem]{Corollary}

\newtheorem{obs}[theorem]{Observation}
\newtheorem{conv}[theorem]{Convention}

\title{\textbf{\sc An Essay on Comp\^{o}nent\u{a} Analysis of Graphs}}

\author{Johan Kok$^1$, Sudev Naduvath$^2$\footnote{Corresponding Author}}
\affil{\small Centre for Studies in Discrete Mathematics\\ Vidya Academy of Science \& Technology \\ Thrissur, Kerala, India.\\{\tt $^1$kokkiek2@tshwane.gov.za},\\{\tt $^2$sudevnk@gmail.com}}

\date{}

\begin{document}
\maketitle

\begin{abstract}
\noindent In most studies related to colouring of graphs and perhaps in the study of other invariants and variants of graphs, the restrictions of non-triviality and connectedness are placed upon graphs. For the introduction to comp\^{o}nent\u{a} analysis, these restrictions are relaxed. In particular, this essay focuses on comp\^{o}nent\u{a} analysis in respect of the recently introduced $J$-colouring. The concept of $J^c$-rainbow connectivity is also introduced in this paper.
\end{abstract}


\vspace{0.25cm}

 
\section{Introduction}

For general notation and concepts in graphs and digraphs see \cite{BM1,FH,DBW}. It is important to note that many results found in this essay are generalised replication of certain results proved in \cite{KS1}. The \textit{order} of a graph $G$ is the number of vertices in $G$ and is denoted by $\nu(G)=n$ and the \textit{size} of $G$ is the number of edges in $G$ and is denoted by $\varepsilon(G)= p$. The \textit{degree} of a vertex $v \in V(G)$ is denoted $d_G(v)$ or when the context is clear, simply as $d(v)$. To begin with, unless mentioned otherwise all graphs $G$ in this section are finite, simple, connected graphs. Hence, triviality is immediately relaxed.

For a sufficiently large value of $\ell\in \N$, let $\mathcal{C}= \{c_1,c_2,c_3,\ldots,c_\ell\}$ is a set of distinct colours. Then, a \textit{proper vertex colouring} of a graph $G$, denoted by $\varphi:V(G)$ is defined as an injective map $\varphi:V(G) \mapsto \mathcal{C}$ such that no two distinct adjacent vertices have the same colour. The cardinality of a minimum set of colours which allows a proper vertex colouring of $G$ is called the \textit{chromatic number} of $G$, denoted $\chi(G)$. When a vertex colouring is considered with colours of minimum subscripts the colouring is called a \textit{minimum parameter colouring}. Unless stated otherwise, we consider minimum parameter colour sets throughout this paper. The number of times a colour $c_i$ is allocated to vertices of a graph $G$ is denoted by $\theta(c_i)$ and $v_i \mapsto c_j$ is abbreviated as $\varphi(v_i) = c_j$. Furthermore, if $c(v_i) = c_j$ then $\iota(v_i) = j$. We also follow the Rainbow Neighbourhood Convention to colour a graph. The rainbow neighbourhood convention is explained in \cite{KSJ} as follows: 

We consider the colouring $(c_1,c_2,c_3,\ldots,c_\ell )$ of a graph $G$, where $\ell=\chi(G)$, in such a way that the colour $c_1$ is assigned to maximum possible number of vertices of $G$, then the colour $c_2$ is given to maximum possible number of remaining vertices among the remaining uncoloured vertices and proceeding like this, at the final step, the remaining uncoloured vertices are given colour $c_\ell$. This convention is called the \textit{rainbow neighbourhood convention}.

The inverse to the rainbow neighbourhood convention requires the mapping $c_j\mapsto c_{\ell -(j-1)}$. Corresponding to the inverse colouring we define $\iota'(v_i) = \ell-(j-1)$ if $c(v_i) = c_j$.  The inverse of a $\chi^-$-colouring is called a $\chi^+$-colouring.

The closed neighbourhood $N[v]$ of a vertex $v \in V(G)$ which contains at least one vertex from each colour class of the chromatic colouring of $G$, is called a \textit{rainbow neighbourhood} (see \cite{KSJ}). In this case, we say that vertex $v$ yields a rainbow neighbourhood. We also say that vertex $u \in N[v]$ belongs to a rainbow neighbourhood.

\begin{definition}\label{Defn-1.1}{\rm 
\cite{NKS} A maximal proper colouring of a graph $G$ is a Johan colouring denoted, $J$-colouring, if and only if every vertex of $G$ belongs to a rainbow neighbourhood of $G$. The maximum number of colours in a $J$-colouring is denoted, $\J(G)$.
}\end{definition}

\begin{definition}\label{Defn-1.2}{\rm 
\cite{NKS} A maximal proper colouring of a graph $G$ is a modified Johan colouring denoted, $J^*$-colouring, if and only if every internal vertex of $G$ belongs to a rainbow neighbourhood of $G$. The maximum number of colours in a $J^*$-colouring is denoted, $\J^*(G)$.
}\end{definition}

\begin{obs}\label{Obs-1.3}
{\rm \cite{NKS}} For any graph $G$ which admits a $J$-colouring, $\chi(G) \leq \J(G)$.
\end{obs}

\section{New Directions}

In most studies related to colouring of graphs and perhaps in the study of other invariants and variants of graphs, the restrictions of non-triviality and connectedness are placed upon graphs. For this introductory study on \textit{comp\^{o}nent\u{a} analysis}, these restrictions are relaxed whenever possible. Throughout  the following discussion, a simple, connected graph will be denoted $G_i,\ i \in \N$. For some $k \in \N$ a graph $G$ is defined to be $G=\bigcup \limits_{i=1}^{k}G_i$. Then , we have the following definition.

\begin{definition}\label{Defn-2.1}{\rm 
A graph $G=\bigcup \limits_{i=1}^{k}G_i$ admits a \textit{comp\^{o}nent\u{a} $J$-colouring}, or \textit{$J^c$-colouring} if and only if each component $G_i$ of $G$ admits a $J$-colouring.
}\end{definition} 

Note that if the maximum set of colours $\mathcal{C}$ corresponds to a $J^c$-colouring of $G$ and the maximum set of colours $\mathcal{C}_i$ corresponds to a $J$-colouring of $G_i$, then $\mathcal{C}_i \subseteq \mathcal{C}$. In terms of the rainbow neighbourhood convention, we say $\mathcal{C}$ provides a subset $J$-colouring to any $G_i$.

\begin{conv}
{\rm \cite{KS1} We consider a null graph on $n$ vertices to be an edgeless graph. It is denoted, $\mathfrak{N}_n$. It follows that, $\J(\mathfrak{N}_n) = \J^*(\mathfrak{N}_n) = 1,\ n \in \N$.
}\end{conv}

\begin{definition}\label{Defn-2.2}{\rm 
The \textit{comp\^{o}nent\u{a} $J$-colouring number} of a graph $G\bigcup \limits_{i=1}^{k}G_i$, denoted by $\J^c(G)$, is defined to be  $\J^c(G)=\max\{\J(G_i):1\leq i\leq k\}$.}
\end{definition}

\begin{definition}\label{Defn-2.3}{\rm 
For a graph $G$ The \textit{comp\^{o}nent\u{a} $J^*$-colouring number} of a graph $G\bigcup \limits_{i=1}^{k}G_i$, denoted by $\J^{*c}(G)$, is defined to be $\J^{*c}(G)=\max\{\J^*(G_i):1\leq i\leq k\}$.}
\end{definition}

\begin{corollary}\label{Cor-2.4}
A graph $G$ has $\J^c(G) = \J(G)$ if and only if $\J(G_i) = \J(G_j),\ 1\leq i,j \leq k$. Similarly, a graph $G$ has $\J^{*c}(G) = \J^*(G)$ if and only if $\J^*(G_i) = \J^*(G_j),\ 1\leq i,j \leq k$.
\end{corollary}
\begin{proof}
The proof is immediate from Definition \ref{Defn-2.2} and Definition \ref{Defn-2.3}.
\end{proof}

We shall now discuss some generalised results of the results provided in \cite{KS1}. Most of these results are presented without further proofs as they are similar to the proofs of the corresponding results established in \cite{KS1}.

\begin{theorem}\label{Thm-2.5}
{\rm \cite{KS1}} For a acyclic graph $G$ of order $n \geq 2$ we have, $\J^c(G) < \J^{*c}(G)$.
\end{theorem}

\begin{corollary}\label{Cor-2.6}
{\rm \cite{KS1}} For any graph $G$ which admits a $J^{*c}$-colouring we have that $\J^{*c}(G) \leq max\{\Delta(G_i):1\leq i \leq k\} + 1$.
\end{corollary}

\begin{corollary}\label{Cor-2.7}
{\rm \cite{KS1}} If $\J^{*c}(G) > \J^c(G)$ for a graph $G$, then at least one $G_i$ for which $\J(G_i) =max\{\J(G_i):1\leq i\leq k\}$, has a pendant vertex.
\end{corollary}

We recall an important lemma from \cite{KS1}.

\begin{lemma}\label{Lem-2.8}{\rm \cite{KS1}}
(i) A maximal proper colouring $\varphi:V(G) \mapsto \mathcal{C}$ of a graph $G$ which satisfies a graph theoretical property say, $\mathfrak{P}$ can be minimised to obtain a minimal proper colouring which satisfies $\mathfrak{P}$.\\
(ii) A minimal proper colouring $\varphi:V(G) \mapsto \mathcal{C}$ of a graph $G$ which satisfies a graph theoretical property say, $\mathfrak{P}$ can be maximised to obtain a maximal proper colouring which satisfies $\mathfrak{P}$.
\end{lemma}

In \cite{KSJ}, the rainbow neighbourhood number $r_\chi(G)$ is defined as the number of vertices of $G$ which yield rainbow neighbourhoods. It is evident that not all graphs admit a $J^c$-colouring. The generalised theorem (see Theorem 2.6 in\cite{KS1}) characterises those graphs which admit a $J^c$-colouring. We can now prove the following theorem.

\begin{theorem}\label{Thm-2.6}
 A graph $G$ of order $n=\sum\limits_{i=1}^{k}n_i,\ n_i=\nu(G_i)$ admits a $J^c$-colouring if and only if $r_{\chi_i}(G_i)=n_i\ \forall i$.
\end{theorem}
\begin{proof}
If $r_{\chi_i}(G_i) = n_i,\ \forall i$, then for all values of $i$, it can be observed that every vertex of $G_i$ belongs to a rainbow neighbourhood in $G_i$. Hence, either the chromatic colouring $\varphi:V(G_i) \mapsto \mathcal{C}_i$ is maximal or a maximal colouring $\varphi': V(G_i)\mapsto \mathcal{C}'_i,\ \forall i$ exists. Also, $\max\{\max\{J(G_i)\}:\text{over all $i$}\}$ exists. Hence, $G$ admits a $J^c$-colouring.

\ni \textit{Converse}: An immediate consequence of definition Definition \ref{Defn-2.3}, is that if graph $G$ admits a $J^c$-colouring then each vertex $v\in V(G_i)$ yields a rainbow neighbourhood in $G_i$. Therefore, $r_{J^c(G)}(G)=n=\sum\limits_{i=1}^{k}n_i$. Hence, from Lemma \ref{Lem-2.8}, it follows that either the $J^c$-colouring is minimal or a minimal colouring $\varphi':V(G_i) \mapsto \mathcal{C}'_i$ exists for each $i$. Also, $\min\{\min\{\J(G_i)\}:\text{over all $i$}\}$ exists. The aforesaid colouring constitutes a $\chi^-$-colouring of $G$ such that $r_{\chi_i}(G_i)=n_i,\ \forall i$. This completes the proof.
\end{proof}

The following theorem characterises a graph which admits a $J^c$-colouring 

\begin{theorem}\label{Thm-2.7}
A graph $G$ admits a $J^c$-colouring if and only if, with respect to some $\chi^-_i$-colouring of $G_i$, each $v \in V(G_i)$, for all $i$, yields a rainbow neighbourhood.
\end{theorem}
\begin{proof}
If  in a $\chi^-_i$-colouring of $G_i$, each $v \in V(G_i)$ yields a rainbow neighbourhood it follows from Lemma \ref{Lem-2.8}(ii) that the corresponding proper colouring can be maximised to obtain a $J$-colouring. Since, $max\{max~\J(G_i):over~all~i\}$ exists, a $J^c$-colouring exists.

Conversely, assume that $G$ admits a $J^c$-colouring.  Then, it implies that for all values of $i$, the graphs $G_i$ admit some $J$-colouring. Therefore, for each $i$ it follows from Lemma 2.5(i) that the corresponding proper colouring can be minimised to obtain a minimal proper colouring for which each $v \in V(G_i)$ yields a rainbow neighbourhood. Let the aforesaid sets of colours be $\mathcal{C}'_i$, $1 \leq i \leq k$.

Without loss of generality, choose any $i,\ 1 \leq i \leq k$. Assume that a minimum set of colours $\mathcal{C}_i$ exists which is a $\chi^-_i$-colouring of $G_i$ and $|\mathcal{C}_i| < |\mathcal{C}'_i|$. It implies that there exists at least one vertex $v \in V(G_i)$ for which at least one distinct pair of vertices, say $u, w \in N(v)$ exists such that $u$ and $v$ are non-adjacent. Furthermore, $c(u)=c(w)$ under $\varphi:V(G_i)\mapsto \mathcal{C}_i$.

Assume that there is exactly one such $v$ and exactly one such vertex pair $u,w \in N(v)$. But, then both $u$, $w$ yield rainbow neighbourhoods in $G_i$ under the proper colouring $\varphi:V(G_i)\mapsto \mathcal{C}_i$ which is a contradiction to the minimality of $\mathcal{C}'_i$. Then, by mathematical induction, similar contradictions arise for all vertices similar to $v$. Therefore the result holds for $G_i$. Induction on $1\leq i \leq k$ follows easily which completes the proof of the general result.
\end{proof}


\section{$J^c$-Rainbow Connectivity in Graphs}

\begin{definition}\label{Def-3.1}{\rm 
A graph $G$ which admits a $J^c$-colouring is said to be \textit{$J^c$-rainbow connected} if for each pair of distinct vertices $u,v \in V(G_i),\ \forall i$, there exists a $(u,v)$-path such that for each colour $c_j\in \mathcal{C}_i$, there exists at least one vertex on the $(u,v)$-path with colour $c_j$. 
}\end{definition}

Note that for $k=1$, the notion $J^c$-connectivity means $J$-connectivity. In such cases, the terminology may be used interchangeably. Similarly, for $k = 1$, a $J$-colouring of $G$ is equivalent to a $J^c$-colouring of $G$. 

\ni Then, we have

\begin{corollary}\label{Cor-3.2}
If a graph $G$ has a component $G_i$ which has $\J(G_i) \geq 3$ and $G_i$ is $J$-connected, then $\delta(G_i) \geq 2$.
\end{corollary}
\begin{proof}
Assume that a component $G_i$ of $G$ has $\J(G_i) \geq 3$ and $G_i$ is $J$-connected. Assume that  $\delta(G_i)=1<2$. Then, at least one pendant vertex say $u$ exists with neighbour, say $v$. Since the $(u,v)$-path is unique the vertices $u,v$ are not $J$-rainbow connected which contradicts the assumption. Therefore, $\delta(G_i)\geq 2$, completing the proof.
\end{proof}

\begin{theorem}\label{Thm-3.3}
If a graph $G$ is $J^c$-rainbow connected then for all pairs of distinct vertices $u,v \in V(G_i)$, $1\leq i \leq k$, there exists a $(u,v)$-path of length at least $\J(G_i)-1$.
\end{theorem}
\begin{proof}
For $G_i = K_{n_i}$ the result is obvious. Without loss of generality, consider any specific $i$, $1\leq i \leq k$ for which $G_i \neq K_{n_i}$. For any pair of distinct vertices $u,v \in V(G_i)$, it is possible to have $c(u) \neq c(v)$ and hence a $(u,v)$-path of length at least $\J(G_i)-1$ must exist to provide for the possibility that each colour appears at least once in such paths. Otherwise, $G_i$ is not $J$-rainbow connected. However, it is possible that $c(u)=c(v)$ and hence a $(u,v)$-path of length at least $J(G_i)$ must exist to provide for the possibility that each colour appears at least once in such path. Hence, if $G_i$ fails to hold the three possible conditions for all values of $i$, then the graph $G$ is not $J^c$-rainbow connected. From the absolute minimum of these permissible minimum path lengths, the result follows immediately.
\end{proof}

\ni The following theorem discusses the $J^c$-rainbow connectivity of forests.

\begin{proposition}\label{Prop-3.4}
A forest $F=\bigcup\limits_{i=1}^{k}T_i$,each $T_i$ being a tree, is $J^c$-rainbow connected. 
\end{proposition}
\begin{proof}
For any $k \in \N$, being the disjoint union of any $k$ trees, the forest $F$ has $\J^c(F)=1$ or $\J^c(F)=2$. But note that only the trivial tree $K_1$ can have $\J^c(F)=1$. Hence, the result follows from the fact that a vertex is inherently adjacent to itself, it is inherently connected to itself. We now further the proof for $T_i$ of order at least $2$. For each pair of distinct vertices $u,v \in V(T_i)$, there exists a $(u,v)$-path such that for each of the colours $c_1$ and $c_2$, there exists at least one vertex $w$ on the $(u,v)$-path such that, colour of $w$ is $c_1$ or $c_2$. Therefore, it follows that a forest $F$ is $J^c$-rainbow connected. 
\end{proof}

\begin{proposition}\label{Prop-3.5}
The graph $G=\bigcup\limits_{i=1}^{k}K_{n_i}$ is $J^c$-rainbow connected. 
\end{proposition}
\begin{proof}
The same reasoning found in the proof of Proposition \ref{Prop-3.4} can be applied to trivial $K_1$ components. Consider only $K_{n_i},\ n_i \geq 2$. Without loss of generality, consider any specific $i$, $1\leq i \leq k$. Obviously, $\J^c(K_{n_i})=n_i$. Since there exists a Hamilton path between each pair of distinct vertices $u,v \in V(K_{n_i})$, there exists a $(u,v)$-path such that for each colour $c_j \in \{c_1, c_2, \ldots, c_{n_i}\}$, there exists at least one vertex $w$ on the $(u,v)$-path such that, $c(w) =c_j$. Therefore, it follows that $K_{n_i}$ is $J$-rainbow connected. Then, by induction, the result follows immediately for the graph $G=\bigcup\limits_{i=1}^{k}K_{n_i}$.
\end{proof}

In \cite{NKS} it was shown that a cycle $C_n$ is $J$-colourable if and only if $n \equiv 0({\rm mod}\ 2)$ or $n\equiv 0({\rm mod}\ 3)$. By similar reasoning found in the proof of Proposition \ref{Prop-3.5}, it follows that

\begin{corollary}\label{Cor-3.6}
A graph $G=\bigcup\limits_{i=1}^{k}C_{n_i}$ is $J^c$-rainbow connected if and only if $C_{n_i}$ are $J$-colourable for $1\leq i \leq k$. 
\end{corollary}
\begin{proof}
The result can be proved using the similar argument in the proof of Proposition \ref{Prop-3.5}.
\end{proof}

In \cite{NKS} it was shown that if a wheel $W_{n+1}$ admits a $J$-colouring then $\J(W_{n+1}) = 4$ if $n \equiv 0({\rm mod}\ 3)$, or $\J(W_{n+1})=3$ if $n\equiv 0({\rm mod}\ 2)$ and $n\not\equiv 0({\rm mod}\ 3)$. Hence, we have 

\begin{proposition}\label{Prop-3.7}
A graph $G=\bigcup\limits_{i=1}^{k}W_{n_i + 1}$ is $J^c$-rainbow connected if and only if $W_{n_i+1}$ are $J$-colourable for $1\leq i \leq k$. 
\end{proposition}
\begin{proof}
If $W_{n_i + 1}$ are $J$-colourable for $1\leq i \leq k$, consider without loss of generality any specific $i$, $1\leq i \leq k$. Let the central vertex be $u_i$ and the cycle vertices be, $v_1,v_2,v_3,\ldots, v_{n_i}$. Since, $u_iv_1v_2v_3 \ldots v_{n_i}u_i$ is a Hamilton cycle of $W_{n_i + 1}$ it follows that a wheel is $J$-rainbow connected. Through immediate induction the result follows for, $G = \bigcup\limits_{i=1}^{k}W_{n_i + 1}$.

\ni The converse part follows trivially from the definition of $J^c$-rainbow connectivity.
\end{proof}

The following result discusses the $J$-rainbow connectivity of a complete $\ell$-partite graph.

\begin{proposition}\label{Prop-3.8}
For $l\le 1$, a complete $l$-partite graph $K_{n_1,n_2,n_3,\ldots,n_l}$ is $J$-rainbow connected.
\end{proposition}
\begin{proof}
For $l\geq 1$, a complete $l$-partite graph $K_{n_1,n_2,n_3,\ldots,n_l}$ is $J$-colourable with $\J(K_{n_1,n_2,n_3,\ldots,n_l})=l$. For $l=1$, we have the null graph $\mathfrak{N}_{n_i}$ and since any vertex is inherently adjacent to itself, it is inherently connected to itself. Therefore, any $P_1$ in respect of a vertex of $\mathfrak{N}_{n_i}$ is the desired path.

For $l\geq 2$, a complete $l$-partite graph, $K_{n_1,n_2,n_3,\ldots,n_l}=(\ldots ((\mathfrak{N}_{n_1}+\mathfrak{N}_{n_2}) + \mathfrak{N}_{n_3})+\ldots + \mathfrak{N}_{n_l})$. Since the complete bipartite graph $\mathfrak{N}_{n_1}+\mathfrak{N}_{n_2}$ is $J$-rainbow connected, it follows through immediate induction that $K_{n_1,n_2,n_3,\ldots,n_l}$ is $J$-rainbow connected.   
\end{proof}

The next result is a direct consequence of Proposition \ref{Prop-3.8}.

\begin{corollary}\label{Cor-3.8}
A graph $G =\bigcup \limits_{i=1}^{k}K_{n_1,n_2,n_3,\ldots,n_{l_i}}$,  is $J^c$-rainbow connected. 
\end{corollary}


Next, we establish an important and interesting theorem which describes a necessary and sufficient condition for a graph which admits a $J^c$-colouring to be $J^c$-rainbow connected.

\begin{theorem}[Fundamental Theorem of $J^c$-Rainbow Connectivity]\label{Thm-3.9}
A graph $G=\bigcup\limits_{i=1}^{k}G_i$ which admits a $J^c$-colouring, (that is, a $J$-colouring for $G_i$, for all values of $i$) is $J^c$-rainbow connected if and only if for some $i$, $\J(G_i) \leq 2$ or each $G_i$, which has a cycle $C_n$ with $n\equiv 0({\rm mod}\ 3)$, has no pendant vertices.
\end{theorem}
\begin{proof}
(i) If $J(G_i) \leq 2$  then $G_i$ is either a trivial component or the adjacent vertices on a $(u,v)$-path must alternate between the colours $c_1,c_2$. Therefore, the result follows immediately.\\ 
(ii) If $G_i$ for some $i$ has a cycle $C_n$, with $n\equiv 0({\rm mod}\ 3)$, then $\J(G_i) \geq 3$. Therefore, $G_i$ cannot have pendant vertices, since otherwise, it cannot admit a $J$-colouring. This contradicts the hypothesis. Hence, the proof is complete.
\end{proof}

\begin{definition}\label{Defn-3.10}{\rm 
A graph $G$ is said to be \textit{$\chi$-rainbow connected} if for each pair of distinct vertices $u,v \in V(G_i)$, $\forall i$, there exists a $(u,v)$-path in $G_i$ such that there exists at least one vertex $w$ on the $(u,v)$-path with colour $c_j$, for each colour $c_j \in \mathcal{C}_i$.
}\end{definition}

The following theorem describes the relation between $J^c$-colouring and $\chi$-rainbow connectivity of a graph $G$.

\begin{theorem}[Fundamental Theorem of $J^c$-Colouring]\label{Thm-3.11}
A graph $G=\bigcup\limits_{i=1}^{k}G_i$ admits a $J^c$-colouring if and only if $G$ is $\chi$-rainbow connected.
\end{theorem}
\begin{proof}
If $G$ is $\chi$-rainbow connected, select any $G_i$. Obviously, $G_i$ is $\chi_i$-rainbow connected. It implies that every vertex $v\in V(G_i)$ has a closed neighbourhood which contains at least one coloured vertex of each colour in the $\chi_i$-colouring set. If the set of colours is not maximal, then Lemma \ref{Lem-2.8}(ii) provides for the existence of such maximal proper colouring. Hence, the existence of a $J_i$-colouring follows and by immediate induction, a $J^c$-colouring is obtained for $G$. 

Conversely, if $G_i,\ \forall i$ admits a $J^c$-colouring, then by Lemma \ref{Lem-2.8}(i) allows for the minimisation of the colour set which results in a $\chi_i$-rainbow connected colouring of $G_i$. Then, by induction, a $\chi$-rainbow connected colouring is obtained for $G$.  This completes the proof.
\end{proof}

\ni Hence, the following result is immediate.

\begin{corollary}\label{Cor-3.12}
A graph $G=\bigcup\limits_{i=1}^{k}G_i$ is $J^c$-rainbow connected if and only if $G$ is $\chi$-rainbow connected.
\end{corollary}

\section{Conclusion}

Clearly, the notion of comp\^{o}nent\u{a} analysis opens a wide scope for further essays on generalising existing results in respect of the numerous variants and invariants of graphs. The authors view comp\^{o}nent\u{a} analysis an easy, but important way to escape from the conventional restrictions placed on graphs in many studies in graph theory. Since, proper colouring of a graph $G$ or derivative colourings thereof are defined in terms of pairs of distinct vertices, a graph need not be simple. Because a vertex is inherently adjacent to itself, a loop expresses explicit adjacency to itself therefore, a graph may have loops. Trivially it may have multiple edges. Therefore, elegant generalisations to include pseudo graphs are open. Many real world problems related to magnetic force fields, gravity force fields and alike can best be modelled as pseudo graphs. The role which colouring of graphs play in these areas remains silent thus far.

\end{document}